\documentclass{article}
\usepackage[affil-it]{authblk}
\usepackage[utf8]{inputenc}
\usepackage{amsmath, amsfonts, amsthm, amssymb, latexsym}
\usepackage{geometry}
\usepackage{enumerate}
\usepackage{multirow}
\usepackage{rotating}
\usepackage{graphicx,color}
\usepackage{float}
\usepackage{mathrsfs}

\usepackage{xargs}
\usepackage[pdftex,dvipsnames]{xcolor}

\newcommand{\dt}{\partial_t}
\newcommand{\ddt}{\partial^2_t}
\newcommand*\diff{\mathop{}\!\mathrm{d}}
\newcommand{\lge}{\langle}
\newcommand{\rge}{\rangle}
\newcommand{\zm}{{Z^m}}
\newcommand{\zmm}{{Z^{m-1}}}

\usepackage{hyperref}
\hypersetup{colorlinks=true, pdfstartview=FitV,linkcolor=blue!70!black,citecolor=red!70!black, urlcolor=green!60!black}
\definecolor{labelkey}{rgb}{0.6,0,0}

\usepackage[colorinlistoftodos,prependcaption,textsize=small]{todonotes}
\newcommandx{\change}[2][1=]{\todo[#1]{#2}}
\newcommandx{\unsure}[2][1=]{\todo[linecolor=red,backgroundcolor=red!25,bordercolor=red,#1]{#2}}
\newcommandx{\rmk}[2][1=]{\todo[linecolor=blue,backgroundcolor=blue!25,bordercolor=blue,#1]{#2}}
\newcommandx{\info}[2][1=]{\todo[linecolor=OliveGreen,backgroundcolor=OliveGreen!25,bordercolor=OliveGreen,#1]{#2}}
\newcommandx{\improvement}[2][1=]{\todo[linecolor=Plum,backgroundcolor=Plum!25,bordercolor=Plum,#1]{#2}}
\newcommandx{\thiswillnotshow}[2][1=]{\todo[disable,#1]{#2}}

\newtheorem{thm}{Theorem}[section]

\newtheorem{prop}[thm]{Proposition}

\theoremstyle{definition}

\theoremstyle{remark}

\def\dt{\partial_t}

\title{An inverse problem for the fractionally damped
	wave equation}

\author[1]{Li Li \thanks{lil33@uci.edu}}
\affil[1]{Department of Mathematics, University of California, Irvine, CA 92697, USA}

\author[2]{Yang Zhang \thanks{yangzh26@uw.edu}}
\affil[2]{Department of Mathematics, University of Washington,
	Seattle, WA 98195, USA}

\date{}

\begin{document}
	
	\maketitle
	
	\noindent \textbf{ABSTRACT.}\, 
	We consider an inverse problem for a Westervelt type nonlinear wave equation with fractional damping.
	This equation arises in nonlinear acoustic imaging and we show the forward problem is locally well-posed.
	We prove that the smooth coefficient of the nonlinearity can be uniquely determined, based on the knowledge of the source-to-solution map and a priori knowledge of the coefficient in an arbitrarily small subset of the domain.
	Our approach relies on a second order linearization  as well as the unique continuation property of the spectral fractional Laplacian.

	\smallskip
	\section{Introduction} \label{Sec:Intro}
	Ultrasound waves are widely used in medical imaging.
	The propagation of high-intensity ultrasound waves are modeled by nonlinear wave equations; see \cite{humphrey2003non}.
	Nonlinear ultrasound waves play an important rule in diagnostic and therapeutic medicine, for example, see \cite{Anvari2015,Demi2014,Harrer2003}.
	Damping effects naturally exist for wave equations in many fields of physics and engineering, for example, see \cite{aanonsen1984distortion}.
	
	In this paper, we consider a nonlinear wave equation of Westervelt type with a damping term, given by 
	\begin{equation}\label{WesterD}
		\partial^2_t (u- \kappa u^2) -\Delta u + D u= f.
	\end{equation}
	Here we focus on the space-fractional damping $D= \partial_t (-\Delta)^s$, which models the case when the damping is frequency-dependent and obeys an empirical power law, see \cite{chen2004fractional} and also \cite{kaltenbacher2021some}.
	The spectral fractional Laplacian $(-\Delta)^s$ ($0< s< 1$) is the fractional power of the Dirichlet Laplacian $-\Delta= (-\Delta)_\Omega$ (the restriction of the Laplacian to the functions satisfying the homogeneous Dirichlet boundary condition on $\partial\Omega$).
	
	More explicitly, let $\Omega$ be a bounded domain with smooth boundary.
	We consider the problem 
	\begin{equation}\label{nonlinearfracdamp}
		\begin{aligned}
			\partial^2_{t} (u- \kappa(x, t)u^2) -\Delta u + \partial_t(-\Delta)^s u&= f,\quad \,\,\,(x, t)\in \Omega\times (0, T),\\
			u & = 0,\quad \,\,\,(x, t)\in \partial \Omega\times (0, T),\\
			u(0)= \partial_t u(0)&= 0,\quad \,\,\,\, x\in \Omega,\\
		\end{aligned}
	\end{equation}
	where $u$ is the pressure field of the acoustic waves and 
	$\kappa\in C^\infty(\bar{\Omega}\times [0, T])$ is the coefficient of the nonlinearity.
	
	We show the local well-posedness of  (\ref{nonlinearfracdamp}) at least for sufficiently regular and small $f$, see Section \ref{Sec:Forward}.  Then for any given nonempty and open $W\subset \Omega$, we can define the  source-to-solution map
	\begin{equation}\label{stos}
		L_{\kappa, W}: f\to u|_{W\times (0, T)},\qquad 
		f\in C^\infty_c(W\times (0, T)).
	\end{equation}
	The goal is to determine $\kappa$ in the whole domain $\Omega\times (0, T)$ based on the knowledge of $L_{\kappa, W}$ and the a priori knowledge of $\kappa$ in $W\times (0, T)$. The following theorem is our main result.
	
	\begin{thm}\label{mthm}
		Suppose $\kappa_1, \kappa_2\in C^\infty(\bar{\Omega}\times [0, T])$ and $\kappa_1= \kappa_2$ in $W\times (0, T)$. Then
		\begin{equation}\label{L12}
			L_{\kappa_1, W}= L_{\kappa_2, W}
		\end{equation}
		implies $\kappa_1= \kappa_2$ in $\Omega\times (0, T)$.
	\end{thm}
	
	\subsection{Connection with earlier literature}
	The inverse problem of determining the nonlinear coefficient from the Dirichlet-to-Neumann map without damping is studied in \cite{ultra21} for the  Westervelt equation and in \cite{UZ_acoustic} for a more general nonlinear model.
	In \cite{eptaminitakis2022weakly}, the authors consider the reconstruction of the nonlinear coefficient using high frequency waves for the Westervelt equation.
	In \cite{zhang2023nonlinear}, the recovery of both a general nonlinearity and a weakly damping term from the Dirichlet-to-Neumann map is studied. 
	The main idea is to use multi-fold linearization and interaction of distorted plane waves. 
	In this case, the nonlinearity helps to solve the inverse problem, as is first shown in \cite{Kurylev2018}. 
	Other damped or attenuated models have been studied in   \cite{kaltenbacher2011well,
		kaltenbacher2012analysis,
		kaltenbacher2021some,
		k2022parabolic,
		kaltenbacher2022simultanenous,
		fu2022inverse,
		kaltenbacher2021identification}. 
	
	The rigorous mathematical study of (Calder\'on type) inverse problems for space-fractional equations was initiated in \cite{ghosh2020calderon} where the authors considered the
	exterior Dirichlet problem
	\begin{equation}\label{eDp}
		((-\Delta)^s+ q)u= 0\quad \mathrm{in}\,\,\Omega,\qquad
		u|_{\Omega_e}= g,
	\end{equation}
	where $(-\Delta)^s$ is the fractional Laplacian in $\mathbb{R}^n$ and $\Omega_e:= \mathbb{R}^n\setminus \Omega$.
	They defined the associated Dirichlet-to-Neumann map
	\begin{equation}\label{DNmapfrac}
		\Lambda_q: g\to (-\Delta)^su|_{\Omega_e}.
	\end{equation}
	Rather than constructing complex geometrical optics solutions (which have been used for solving the classical Calder\'on problem), 
	the authors exploited the nonlocal features of the fractional operator to uniquely determine the potential $q$ in $\Omega$ from partial knowledge of $\Lambda_q$. We refer readers to \cite{covi2023reduction,covi2022global, kow2022calderon,lai2023inverse,li2020determining,li2023elas,li2023porous,li2023inverse,lin2023strong} for further unique determination results for fractional operators based on the knowledge of the Dirichlet-to-Neumann map.
	
	Besides, there are also several unique determination results for fractional operators based on the knowledge of the source-to-solution map
	in the existing literature. In \cite{feizmohammadi2021fractional, quan2022calder, chien2022inverse}, the authors considered equations involving spectral fractional operators on manifolds, and they determined the Riemannian manifold up to an isometry.
	
	In this paper, we combine the elements in \cite{ghosh2020calderon} and \cite{feizmohammadi2021fractional} in the setting of fractional damping.
	Our source-to-solution map (\ref{stos}) can be viewed as an analogue of (1.2) in \cite{feizmohammadi2021fractional}. Our approach to proving the unique determination result is motivated by the framework established in \cite{ghosh2020calderon}. We will see that nonlocal phenomenons will play a fundamental role in solving the inverse problem as expected.

	\subsection{Organization}
	The rest of this paper is organized in the following way. In Section 2, we will summarize the background knowledge. In Section 3, we will first show the well-posedness of a linear problem associated with (\ref{nonlinearfracdamp}) and obtain several regularity results. Then we will further use a fixed-point argument to show the well-posedness of (\ref{nonlinearfracdamp}) for small $f$. In Section 4, we will first prove the unique continuation property of the spectral fractional Laplacian and derive the related Runge approximation property. Then we will combine the unique continuation property and the Runge approximation property with a second order linearization technique to prove the main theorem.
	
	\medskip
	
	\noindent \textbf{Acknowledgements.} L.L. and Y.Z. would like to thank Professor Katya Krupchyk and Professor Gunther Uhlmann for helpful discussions.

	\smallskip
	\section{Preliminaries} \label{Sec:Prelim}
	
	Throughout this paper we use the following notations.
	
	\begin{itemize}
		
		\item We fix the space dimension
		$n=3$. (All arguments in the inverse problem section work for the general $n$. Our restriction $n= 3$ is only used in Proposition \ref{emb}
		and Proposition \ref{linkpest}.)
		
		\item We fix the fractional power $0< s< 1$ and the length of the time interval $T> 0$.
		
		\item $\Omega$ denotes a bounded domain with smooth boundary.
		
		\item $c, C, C', C_1,\cdots$ denote positive constants (which may depend on some parameters).
		
		\item $\langle \cdot, \cdot\rangle$ denotes the standard $L^2$-distributional pairing.
		
	\end{itemize}
	
	\subsection{Sobolev spaces and fractional operators}
	We use $H^s$ to denote the standard $W^{s,2}$-type Sobolev space.
	
	Let $U$ be an open set in $\mathbb{R}^n$. Let $F$ be a closed set in $\mathbb{R}^n$. Then 
	$$H^s(U):= \{u|_U: u\in H^s(\mathbb{R}^n)\},\qquad 
	H^s_F(\mathbb{R}^n):= 
	\{u\in H^s(\mathbb{R}^n): \mathrm{supp}\,u\subset F\},$$
	$$\tilde{H}^s(U):= 
	\mathrm{the\,\,closure\,\,of}\,\, C^\infty_c(U)\,\,\mathrm{in}\,\, H^s(\mathbb{R}^n).$$
	
	Since $\Omega$ is a bounded domain with smooth boundary, we have the identification
	$\tilde{H}^s(\Omega)= H^s_{\bar{\Omega}}(\mathbb{R}^n)$, and its dual space is
	$H^{-s}(\Omega)$.
	
	The Dirichlet Laplacian $-\Delta$ is a non-negative self-adjoint operator in $\tilde{H}^1(\Omega)$. Therefore there exists an orthonormal
	basis of $L^2(\Omega)$ consisting of eigenfunctions $\phi_k\in \tilde{H}^1(\Omega)$ ($k= 1,2,\cdots,$) that correspond to the eigenvalues $0< \lambda_1\leq \lambda_2\leq\cdots$. The domain of the spectral fractional Laplacian is defined by
	$$\mathrm{Dom}((-\Delta)^s):=
	\{u\in L^2(\Omega): \sum_{k=1}^\infty \lambda_k^{s} |\langle u, \phi_k \rangle|^2 < +\infty\}.$$
	
	It is well-known (see Subsection 2.1 in \cite{caffarelli2016fractional} and Subsection 3.1.3 in \cite{bonforte2014existence}) that $\mathrm{Dom}((-\Delta)^s)= \tilde{H}^s(\Omega)$ and the spectral fractional Laplacian mapping $\tilde{H}^s(\Omega)$ into 
	$H^{-s}(\Omega)$ is defined by
	\begin{equation}\label{specfrac}
		(-\Delta)^s u:=\sum_{k=1}^\infty \lambda_k^{s} \langle u, \phi_k \rangle\phi_k.
	\end{equation}
	
	The spectral fractional Laplacian can be also equivalently defined via the semigroup approach (see Lemma 2.2 in \cite{caffarelli2016fractional}). Let $U(x, t)= e^{-t(-\Delta)}u(x)$ be the solution of the parabolic problem
	\begin{equation}
		\left\{
		\begin{aligned}
			\partial_{t} U -\Delta U &= 0,\quad \,\,\,(x, t)\in \Omega\times (0, \infty),\\
			U&= 0,\quad \,\,\, (x,t)\in \partial\Omega\times (0, \infty),\\
			U|_{t=0}&= u,\quad \,\,\,x\in\Omega.\\
		\end{aligned}
		\right.
	\end{equation}
	Then for $u\in \tilde{H}^s(\Omega)$,
	\begin{equation}\label{semifracdef}
		(-\Delta)^s u= \frac{1}{\Gamma(-s)}\int^\infty_0(U- u)\frac{\diff t}{t^{1+s}}
	\end{equation}
	in $H^{-s}(\Omega)$, where $\Gamma(\cdot)$ is the standard Gemma function. More precisely, for 
	$v\in \tilde{H}^s(\Omega)$, we have
	\begin{equation}\label{semifracL2}
		\langle (-\Delta)^s u, v\rangle= \frac{1}{\Gamma(-s)}\int^\infty_0(\langle U, v\rangle- \langle u, v\rangle)\frac{\diff t}{t^{1+s}}.
	\end{equation}
	We remark that the spectral fractional Laplacian defined here is different from the restriction of 
	$(-\Delta_{\mathbb{R}^n})^s$ to $\Omega$, although they enjoy several common properties (see Subsection 2.1 in \cite{bonforte2014existence}).
	
	\subsection{Set Z}
	To study the well-posedness of (\ref{nonlinearfracdamp}), we introduce the set $\zm(R,T)$ consisting of $u$ satisfying
	$$u \in \bigcap_{k=0}^{m} H^{m-k}(0,T; H^k(\Omega)),
	\qquad \|u\|^2_{\zm} = \sum_{k=0}^m \int_{0}^T   \|\partial_t^{m-k} u(t)\|^2_{H^k} \diff t \leq R^2.$$
	
	The proof of Claim 1 in \cite{Uhlmann2021a} ensures that
	$\zm(R,T)$ has the following property.
	
	\begin{prop}\label{normineq}
		Suppose $u \in Z^m(R, T)$ for some $R>0$.
		Then $\|u\|_\zmm \leq \|u\|_\zm$ and $\dt u \in Z^{m-1}(R, T)$ with $\|\dt u\|_\zmm \leq \|u\|_\zmm$. Moreover, we have the following estimates.
		\begin{enumerate}[(1)]
			\item If $v \in Z^m(R', T)$, then $\|uv\|_\zm \leq C \|u\|_\zm \|v\|_\zm$.
			\item If $v \in \zmm(R', T)$, then $\|uv\|_\zmm \leq C \|u\|_\zm \|v\|_\zmm$.
		\end{enumerate}
	\end{prop}

	\smallskip
	\section{Forward problem} \label{Sec:Forward}
	
	\subsection{Linear equation}
	We first study the well-posedness of the linear problem 
	\begin{equation}\label{lin}
		\left\{
		\begin{aligned}
			\partial^2_{t} u -\Delta u + \partial_t(-\Delta)^s u&= f,\quad \,\,\,(x, t)\in\Omega\times (0, T),\\
			u(0)= \partial_t u(0)&= 0,\quad \,\,\,x\in\Omega.\\
		\end{aligned}
		\right.
	\end{equation}
	
	\begin{prop}\label{linwell}
		For any $f\in L^2(0, T; L^2(\Omega))$, (\ref{lin}) has a unique solution $u$ satisfying
		$$u\in H^2(0, T; H^{-1}(\Omega))\cap W^{1,\infty}(0, T; L^2(\Omega))\cap L^\infty(0, T; \tilde{H}^1(\Omega)),\qquad \partial_t u\in 
		L^2(0, T; \tilde{H}^s(\Omega)).$$ 
		Moreover, for $t\in [0, T]$, we have the estimate
		\begin{equation}\label{linest}
			\|\dt u(t)\|_{L^2}^2 + \|\nabla u(t)\|_{L^2}^2
			+ \int_0^t \| \dt (-\Delta)^{s/2}u(\tau) \|_{L^2}^2 \diff \tau
			\leq  C\int_{0}^t \|f(\tau)\|_{L^2}^2 \diff \tau.
		\end{equation}
	\end{prop}
	
	\begin{proof}
		We use the Galerkin method. For $l\in \mathbb{N}$,
		consider the approximate solution $u_l(t)$ of the form $\sum_{k=1}^l u_{l,k}(t) \phi_k$ satisfying
		\[
		\lge \partial^2_{t} u_l, v \rge + \lge \nabla u_l, \nabla v \rge + \lge \partial_t (-\Delta)^s u_l, v \rge
		= \lge f, v \rge
		\]
		for any $v$ in the space spanned by $\phi_1, \ldots, \phi_l$ and the initial conditions $u_l(0)= \partial_t u_l(0)= 0$. (The standard theory for linear ODE 
		systems ensures that $C^2$-function $u_{l,k}$ can be uniquely determined.)
		
		By choosing $v = \dt u_l$, we have
		$$\frac{1}{2}(\frac{\diff }{\diff t} \|\dt u_l(t)\|_{L^2}^2 + \frac{\diff }{\diff t} \|\nabla u_l(t)\|_{L^2}^2) + \|\dt (-\Delta)^{s/2} u_l(t)\|_{L^2}^2
		= \lge f, \dt u_l \rge,$$
		$$\frac{1}{2}( \|\dt u_l(t)\|_{L^2}^2
		+ \|\nabla u_l(t)\|_{L^2}^2 ) + \int_0^t\|\dt (-\Delta)^{s/2} u_l(\tau)\|_{L^2}^2\diff \tau\\
		= \int_0^t \lge f(\tau), \dt u_l(\tau) \rge \diff \tau.
		$$
		Since the first eigenvalue the Dirichlet Laplacian is strictly positive, the definition of the spectral fractional Laplacian (\ref{specfrac}) ensures the Poincar\'e inequality
		$$\|(-\Delta)^{s/2} v\|_{L^2}^2\geq 
		c\|v\|_{L^2}^2,\qquad v\in \tilde{H}^s(\Omega).$$
		Using the inequality
		$$\int_0^t \lge f(\tau), \dt u_l(\tau) \rge \diff \tau\leq \frac{2}{c}\int_0^t\|f(\tau)\|_{L^2}^2 \diff \tau
		+ \frac{c}{2}\int_0^t\|\dt u_l(\tau)\|_{L^2}^2 \diff \tau,$$
		we obtain
		$$\|\dt u_l(t)\|_{L^2}^2
		+ \|\nabla u_l(t)\|_{L^2}^2  + \int_0^t\|\dt (-\Delta)^{s/2} u_l(\tau)\|_{L^2}^2\diff \tau\\
		\leq \frac{4}{c}\int_0^t\|f(\tau)\|_{L^2}^2 \diff \tau.
		$$
		Thus, $\{u_l\}^\infty_{l=1}$ is bounded in 
		$$H^2(0, T; H^{-1}(\Omega))\cap W^{1,\infty}(0, T; L^2(\Omega))\cap L^\infty(0, T; \tilde{H}^1(\Omega))$$
		and $\{\partial_t u_l\}^\infty_{l=1}$ is bounded in $L^2(0, T; \tilde{H}^s(\Omega))$. By using the standard compactness argument, we can find a subsequence of 
		$\{u_l\}^\infty_{l=1}$ weakly convergent to $u$ satisfying (\ref{lin}) and (\ref{linest}). The uniqueness of the solution directly follows from
		the estimate (\ref{linest}).
	\end{proof}
	
	To study the well-posedness of (\ref{nonlinearfracdamp}) later, we also need to consider the following linear problem 
	\begin{equation}\label{linkp}
		\left\{
		\begin{aligned}
			(1- 2\kappa v)\partial^2_{t} u -\Delta u + \partial_t(-\Delta)^s u&= f,\quad \,\,\,(x, t)\in\Omega\times (0, T),\\
			u(0)= \partial_t u(0)&= 0,\quad \,\,\,x\in\Omega.\\
		\end{aligned}
		\right.
	\end{equation}
	
	\begin{prop}\label{linkpest}
		Suppose $m \geq 8$, $\kappa \in C^\infty(\bar{\Omega}\times [0, T])$ and $f \in \zmm(R, T)$. If $v \in \zmm(r_0, T)$ for some small constant $r_0> 0$, then (\ref{linkp}) has a unique solution $u$ satisfying 
		$$u \in \bigcap_{k=0}^{m} H^{m-k}(0,T; H^k(\Omega)),
		\qquad \|u\|_{\zm} \leq C \|f\|_{\zmm}.$$
	\end{prop}
	
	We can still use the Galerkin method to prove the proposition above, but we need to derive more complicated energy estimates. We refer readers to Appendix for more details.
	
	\subsection{Nonlinear equation}
	Based on Proposition \ref{linkpest}, we can use a fixed-point argument to show the well-posedness of (\ref{nonlinearfracdamp}) for small $f$.
	
	\begin{prop}\label{nonlinkpest}
		Suppose $m \geq 8$, $\kappa \in C^\infty(\bar{\Omega}\times [0, T])$ and $f \in \zmm(\rho, T)$. Then for sufficiently small $\rho> 0$,
		(\ref{nonlinearfracdamp}) has a unique solution $u$ satisfying 
		$$u \in \bigcap_{k=0}^{m} H^{m-k}(0,T; H^k(\Omega)),
		\qquad \|u\|_{\zm} \leq C \|f\|_{\zmm}.$$
	\end{prop}
	\begin{proof}
		We consider the problem 
		\begin{equation}
			\left\{
			\begin{aligned}
				(1- 2\kappa v)\partial^2_{t} u -\Delta u + \partial_t(-\Delta)^s u&= f+ 2\kappa(\partial_t v)^2+ 4(\partial_t\kappa)v\partial_t v+ (\partial^2_t\kappa)v^2,\quad \,\,\,(x, t)\in\Omega\times (0, T),\\
				u(0)= \partial_t u(0)&= 0,\quad \,\,\,x\in\Omega.\\
			\end{aligned}
			\right.
		\end{equation}
		For given $f \in \zmm(\rho, T)$, we consider the map
		$$J: v\to u,\qquad v\in Z^m(r, T).$$
		The parameters $\rho, r< r_0$ will be chosen later to ensure $J$ is a contraction map on 
		$Z^m(r, T)$.
		
		In fact, by Proposition \ref{normineq} and Proposition \ref{linkpest}, we have
		$$\|f+ 2\kappa(\partial_t v)^2+ 4(\partial_t\kappa)v\partial_t v+ (\partial^2_t\kappa)v^2\|_{\zmm}\leq 
		\|f\|_{\zmm}+ C_{\kappa}\|v\|^2_{\zm},$$
		\begin{equation}\label{estur}
			\|u\|_{\zm}\leq C'_{\kappa}(\|f\|_{\zmm}+ \|v\|^2_{\zm})\leq C'_{\kappa}(\rho+ r^2).
		\end{equation}
		Hence, we can choose $r < \min\{1,\, 1/{(2C'_{\kappa})}$\} and $\rho = r/(2C'_{\kappa})$ to ensure that $J$ maps $Z^m(r, T)$ into itself.
		Now let $u_j= Jv_j$ for $v_j\in Z^m(r, T)$ ($j= 1,2$). We write $w:= u_2- u_1$. Then $w$ satisfies
		$$(1- 2\kappa v_1)\partial^2_{t} w -\Delta w + \partial_t(-\Delta)^s w= 2\kappa (\dt v_1 + \dt v_2)(\dt v_2 - \dt v_1) + 2\kappa(v_2 - v_1) \ddt u_2$$
		$$+ 4(\partial_t\kappa)(v_2- v_1)\partial_t v_1
		+ 4(\partial_t\kappa)v_2\partial_t(v_2-v_1)
		+(\partial^2_t\kappa)(v_1+ v_2)(v_1- v_2).$$
		By Proposition \ref{normineq}, the $\zmm$-norm of the right hand side is bounded above by
		$$C'(\|v_1\|_{\zm}+ \|v_2\|_{\zm}+ ||u_2||_{\zm})\|v_1- v_2\|_{\zm}.$$
		Since $v_j\in Z^m(r, T)$ implies $u_j\in Z^m(r, T)$,
		by Proposition \ref{linkpest} we have
		$$\|u_2- u_1\|_{\zm}\leq C''r\|v_1- v_2\|_{\zm}.$$
		Hence, $J$ is a contraction map for small $r$.
		
		Note that the fixed point $u= v$ of $J$ is the solution of (\ref{nonlinearfracdamp}). Also note that for small $r$ and $v\in Z^m(r, T)$, we have
		$C'_{\kappa}\|v\|_{\zm}< \frac{1}{2}$
		in (\ref{estur}). This implies $\|u\|_{\zm} \leq C \|f\|_{\zmm}$.
	\end{proof}
	
	\smallskip
	\section{Inverse problem} \label{Sec:IP}
	
	\subsection{Unique continuation property}
	The following proposition is an analogue of the 
	unique continuation property of the fractional Laplacian in $\mathbb{R}^n$, which was first established in \cite{ghosh2020calderon} based on the
	the Carleman estimates in \cite{ruland2015unique}. Here we will exploit the semigroup definition of the spectral fractional Laplacian (\ref{semifracdef}) and the unique continuation property of the classical parabolic operator in the proof. This idea was also used for proving Theorem 1.1 in \cite{feizmohammadi2021fractional} and Proposition 3.2 in \cite{ghosh2021calder}.
	
	\begin{prop}\label{UCP}
		Let $u\in \tilde{H}^s(\Omega)$. Suppose 
		$$(-\Delta)^s u= u= 0$$ in $W$. Then $u= 0$ in $\Omega$.
	\end{prop}
	\begin{proof} Based on the semigroup definition (\ref{semifracdef}), the assumption implies 
		$$\int^\infty_0 \frac{U(x, t)}{t^{1+s}}\,dt= 0,\qquad x\in W,$$
		where 
		$$U(x, t):=  e^{-t(-\Delta)}u(x)= \int_\Omega p_t(x, y) u(y)\diff y$$
		and $p_t(x, y)$ is the heat kernel associated with the Dirichlet Laplacian. The integral here (and all integrals below) should be interpreted in the distributional sense, i.e.
		$$\int^\infty_0 \frac{\langle U(t), \phi\rangle}{t^{1+s}}\,dt= 0,\qquad \phi\in C^\infty_c(W).$$
		Based on the heat kernel estimate (see Corollary 3.2.8 in \cite{davies1989heat})
		$$p_t(x, y)\leq Ct^{-\frac{n}{2}}e^{-\frac{|x-y|^2}{ct}},\qquad x,y\in \Omega,\,t> 0.$$
		
		Now we fix a nonempty $W'\subset \subset W$. Let $c'= \mathrm{dist}(W', \Omega\setminus W)$. Then we have
		\begin{equation}\label{ptxyest}
			p_t(x, y)\leq Ct^{-\frac{n}{2}}e^{-\frac{(c')^2}{ct}}
		\end{equation}
		for $x\in W', y\in \Omega\setminus W$. For $m\in \mathbb{N}$, 
		we will inductively show that
		\begin{equation}\label{Uint}
			\int^\infty_0 \frac{U(x, t)}{t^{m+s}}\,dt= 0,\qquad x\in W'.
		\end{equation}
		
		In fact, once we have shown the case $m$, we apply $-\Delta$ to (\ref{Uint}). Since $U$ solves the heat equation,
		we have
		\begin{equation}\label{tUint}
			\int^\infty_0 \frac{\partial_t U(x, t)}{t^{m+s}}\,dt
			=\int^\infty_0 \frac{\Delta U(x, t)}{t^{m+s}}\,dt= 0,\qquad x\in W'.
		\end{equation}
		Note that for $\phi\in C^\infty_c(W)$, by (\ref{ptxyest}) we have
		$$\langle U(t), \phi\rangle= \int_{W'}\int_\Omega
		p_t(x, y)u(y)\phi(x)\diff y\diff x
		= \int_{W'}\int_{\Omega\setminus W}
		p_t(x, y)u(y)\phi(x)\diff y\diff x\leq  C't^{-\frac{n}{2}}e^{-\frac{(c')^2}{ct}},$$
		so $\frac{\langle U(t), \phi\rangle}{t^{m+s}}$ vanishes at both $0$ and $+\infty$.
		Hence we can integrate by parts to derive
		\begin{equation}\label{tmUint}
			\int^\infty_0 \frac{U(x, t)}{t^{m+1+s}}\,dt= 0,\qquad x\in W'
		\end{equation}
		from (\ref{tUint}). Hence we have verified (\ref{Uint}).
		
		Now we consider the substitution $\lambda= \frac{1}{t}$ and define
		$$V(x, \lambda):= 1_{(0, \infty)}(\lambda)\frac{U(x, \frac{1}{\lambda})}{\lambda^{-s}}.$$
		Then (\ref{Uint}) becomes
		\begin{equation}\label{Vint}
			\int_{\mathbb{R}}V(x, \lambda)\lambda^{m-1}\diff \lambda= 0,\qquad x\in W',\,\, m\in\mathbb{N}.
		\end{equation}
		Note that for each $\phi\in C^\infty_c(W)$, the function
		$$\int_{\mathbb{R}}\langle V(\lambda), \phi\rangle e^{-i\xi\lambda}\diff \lambda$$
		is holomorphic for $\mathrm{Im}\xi< \frac{(c')^2}{c}$, and (\ref{Vint}) implies all its derivatives at $\xi= 0$ are zeros. 
		
		Hence we conclude that the Fourier transform of $\langle V(\lambda), \phi\rangle$ is zero for $\xi\in\mathbb{R}$, so
		$U(x, \frac{1}{\lambda})= U(x, t)= 0$ in $W'\times (0, \infty)$. By the  unique
		continuation property of the classical parabolic operator (see \cite{vessella2003carleman}), we conclude that $U= 0$ in $\Omega\times (0, \infty)$ and thus $u= 0$ in $\Omega$.
		
	\end{proof}
	
	\subsection{Runge approximation property}
	Now we prove a Runge approximation property based on the unique continuation property of the spectral fractional Laplacian and the well-posedness of (\ref{lin}) (and its dual problem). The following proposition can be viewed as a variant of the 
	Runge approximation properties for evolutionary fractional operators established in \cite{ruland2020quantitative, li2021fractional, li2023inverse}.
	
	\begin{prop}\label{RAP}
		The set
		$$S:=\{u_f|_{(0, T)\times(\Omega\setminus W)}: f\in C^\infty_c(W\times (0, T))\}$$
		is dense in $L^2(0, T; L^2(\Omega\setminus W))$. Here $u_f$ is the solution of (\ref{lin}) corresponding to the source $f$.
	\end{prop}
	\begin{proof}
		By the Hahn-Banach Theorem, it suffices to prove the following statement: 
		
		Let $g\in L^2(0, T; L^2(\Omega\setminus W))$. If 
		$$\int^T_0\int_{\Omega\setminus W}ug= 0$$ for all $u\in S$
		then $g= 0$. 
		
		We consider $\tilde{g}\in L^2(0, T; L^2(\Omega))$ which is the zero extension of $g$, and the dual problem
		\begin{equation}\label{duallin}
			\left\{
			\begin{aligned}
				\partial^2_{t} v -\Delta v - \partial_t(-\Delta)^s v&= \tilde{g}\quad \,\,\,\Omega\times (0, T)\\
				v(T)= \partial_t v(T)&= 0\quad \,\,\,\Omega.\\
			\end{aligned}
			\right.
		\end{equation}
		Proposition \ref{linwell} ensures that this problem has a unique solution $v$ satisfying
		$$v\in H^2(0, T; H^{-1}(\Omega))\cap W^{1,\infty}(0, T; L^2(\Omega))\cap L^\infty(0, T; \tilde{H}^1(\Omega)),\qquad \partial_t v\in 
		L^2(0, T; \tilde{H}^s(\Omega)).$$ 
		The assumption implies 
		\begin{equation}\label{RAP1}
			0= \int^T_0\langle \partial^2_{t} v -\Delta v -(-\Delta)^s v, u\rangle\diff t
		\end{equation}
		for $u\in S$. Based on the initial and final conditions, we integrate by parts to obtain
		$$-\int^T_0\langle \partial_t(-\Delta)^s v, u\rangle\diff t
		= \int^T_0\langle \partial_t(-\Delta)^s u, v\rangle\diff t,\quad \int^T_0\langle \partial^2_{t} v, u\rangle\diff t=
		\int^T_0\langle \partial^2_{t} u, v\rangle\diff t.$$
		Hence (\ref{RAP1}) implies 
		$$0=\int^T_0\int_{\Omega}fv\diff x\diff t= \int^T_0\int_{W}fv\diff x\diff t$$
		for $f\in C^\infty_c(W\times (0, T))$ since $u$ is the solution of (\ref{lin}). Hence $v= 0$ in $W\times (0, T)$.
		Note that 
		$$\partial^2_{t} v -\Delta v -\partial_t(-\Delta)^s v= 0$$
		in $W\times (0, T)$ since $u$ is the solution of (\ref{duallin}), so $\partial_t(-\Delta)^s v= 0$ in $W\times (0, T)$. By Proposition \ref{UCP}, we have $\partial_t v= 0$ in $\Omega\times (0, T)$. We further conclude that $v= 0$
		in $\Omega\times (0, T)$ based on the final conditions and thus $g= 0$.
	\end{proof}
	
	\subsection{Proof of the main theorem}
	We are ready to prove Theorem \ref{mthm}.
	
	Our proof will heavily rely on the unique continuation property (Proposition
	\ref{UCP}) of the spectral fractional Laplacian and the Runge approximation property (Proposition \ref{RAP}) associated with (\ref{lin}), which are typical nonlocal phenomenons. To relate the nonlinear (\ref{nonlinearfracdamp}) to the linear (\ref{lin}), we will perform a second order linearization. We remark that this kind of multiple-fold linearizations have been
	widely applied in solving inverse problems for nonlinear equations (see for instance, \cite{Kurylev2018, krupchyk2020remark, uhlmann2021inverse}).
	
	\begin{proof}
		For $f_1, f_2\in C^\infty_c(W\times (0, T))$, we use
		$u^{(j)}_{\epsilon_1, \epsilon_2}$ to denote the solution of 
		\begin{equation}\label{equ12}
			\left\{
			\begin{aligned}
				\partial^2_{t} (u- \kappa_j(x, t)u^2) -\Delta u + \partial_t(-\Delta)^s u&= \epsilon_1f_1+ \epsilon_2f_2,\quad \,\,\,(x, t)\in\Omega\times (0, T),\\
				u(0)= \partial_t u(0)&= 0,\quad \,\,\,x\in\Omega\\
			\end{aligned}
			\right.
		\end{equation}
		($j= 1,2$) for small $\epsilon_1, \epsilon_2$. 
		Then 
		$$\frac{\partial}{\partial\epsilon_j}|_{\epsilon_j= 0}u^{(1)}_{\epsilon_1, \epsilon_2}=\frac{\partial}{\partial\epsilon_j}|_{\epsilon_j= 0}u^{(2)}_{\epsilon_1, \epsilon_2}=: w_j$$
		is the solution of 
		\begin{equation}
			\left\{
			\begin{aligned}
				\partial^2_{t} w -\Delta w + \partial_t(-\Delta)^s  w&= f_j,\quad \,\,\,(x, t)\in\Omega\times (0, T),\\
				w(0)= \partial_t w(0)&= 0,\quad \,\,\,x\in\Omega.\\
			\end{aligned}
			\right.
		\end{equation}
		
		Let $v^{(j)}= \frac{\partial^2}{\partial\epsilon_1\partial\epsilon_2}|_{\epsilon_1= \epsilon_2= 0}u^{(j)}_{\epsilon_1, \epsilon_2}$. Then we have
		\begin{equation}\label{vw12}
			\left\{
			\begin{aligned}
				\partial^2_{t} v^{(j)} -\Delta v^{(j)} -2\partial^2_{t}(\kappa_j(x, t)w_1w_2)+ \partial_t(-\Delta)^s v^{(j)}&= 0,\quad \,\,\,(x, t)\in\Omega\times (0, T),\\
				v^{(j)}(0)= \partial_t v^{(j)}(0)&= 0,\quad \,\,\,x\in\Omega.\\
			\end{aligned}
			\right.
		\end{equation}
		
		The assumption (\ref{L12}) implies
		$$u^{(1)}_{\epsilon_1, \epsilon_2}- u^{(2)}_{\epsilon_1, \epsilon_2}= 0
		$$
		in $W\times (0, T)$.
		Then the assumption $\kappa_1= \kappa_2$ in $W\times (0, T)$ implies 
		$$\partial^2_{t} (u^{(1)}_{\epsilon_1, \epsilon_2}- \kappa_1(x, t)(u^{(1)}_{\epsilon_1, \epsilon_2})^2) -\Delta u^{(1)}_{\epsilon_1, \epsilon_2}=
		\partial^2_{t} (u^{(2)}_{\epsilon_1, \epsilon_2}- \kappa_2(x, t)(u^{(2)}_{\epsilon_1, \epsilon_2})^2) -\Delta u^{(2)}_{\epsilon_1, \epsilon_2}$$
		in $W\times (0, T)$.
		By the equation in (\ref{equ12}) we have
		$$(-\Delta)^s\partial_t(u^{(1)}_{\epsilon_1, \epsilon_2}- u^{(2)}_{\epsilon_1, \epsilon_2})= 0$$
		in $W\times (0, T)$. By Proposition \ref{UCP} we conclude that 
		$\partial_t u^{(1)}_{\epsilon_1, \epsilon_2}= \partial_t u^{(2)}_{\epsilon_1, \epsilon_2}$ in $\Omega\times (0, T)$.
		Then the initial conditions imply $u^{(1)}_{\epsilon_1, \epsilon_2}= u^{(2)}_{\epsilon_1, \epsilon_2}$ in $\Omega\times (0, T)$ and thus 
		$v^{(1)}= v^{(2)}$ in $\Omega\times (0, T)$.
		
		Now by the equation in (\ref{vw12}) we have
		\begin{equation}\label{k12w}
			\partial^2_{t}((\kappa_1(x,t)-\kappa_2(x,t))w_1w_2)= 0
		\end{equation}
		in $\Omega\times (0, T)$. We choose $\phi\in C^\infty(\mathbb
		{R}^n)$ s.t. $\mathrm{supp}\,\phi\subset \bar{\Omega}$ and $\phi> 0$ in $\Omega$. Let $\tilde{\phi}(x, t):= (t- T)^2\phi(x)$. Then 
		$\tilde{\phi}(T)= \partial\tilde{\phi}(T)= 0$.
		Let (\ref{k12w}) act on $\tilde{\phi}$. Based on the initial and final conditions, we can integrate by parts to obtain 
		$$\int^T_0\int_{\Omega\setminus W}w_1(x, t)w_2(x, t)(\kappa_1(x, t)-\kappa_2(x, t))\phi(x)\diff x\diff t= 0.$$
		By Proposition \ref{RAP}, we can choose $f_1, f_2\in C^\infty_c(W\times (0, T))$ s.t. 
		$$w_1\to 1,\qquad w_2\to \kappa_1(x, t )-\kappa_2(x, t)$$ in
		$L^2(0, T; L^2(\Omega\setminus W))$.
		Then we take the limit to obtain
		$$\int^T_0\int_{\Omega\setminus W}(\kappa_1(x, t)-\kappa_2(x, t))^2\phi(x)\diff x\diff t= 0.$$
		Hence we conclude that $\kappa_1(x, t)=\kappa_2(x, t)$ in $(\Omega\setminus W)\times (0, T)$ and thus in $\Omega\times (0, T)$.
	\end{proof}
	
	\appendix
	\section{Appendix}
	To prove Proposition \ref{linkpest}, we need the following embedding result, which follows from the Sobolev embedding  $H^k(\Omega)\hookrightarrow C^{k-2}(\bar{\Omega})$(where we use the assumption $n=3$) and Theorem 2 in Subsection 5.9.2 in \cite{evans2022partial}.
	\begin{prop}\label{emb}
		Suppose $l, k$ are positive integers and $k\geq 2$. If $u \in H^l(0,T; H^k (\Omega))$,
		then
		\[
		u \in C^{l-1}([0,T]; H^k(\Omega)) \quad \text{and} \quad u \in C^{l-1}([0,T]; C^{k-2}(\bar{\Omega}))
		\]
		with the estimates
		\[
		\sum_{j = 0}^{l-1} \sup_{t \in [0, T]}\| \dt^{j} u(t) \|_{C^{k-2}(\bar{\Omega})} \leq C\sum_{j = 0}^{l-1} \sup_{t \in [0, T]}\| \dt^{j} u(t) \|_{H^k(\Omega)}
		\leq C'\| u\|_{H^{l}(0,T; H^k (\Omega))}.\]
	\end{prop}
	
	Now we provide an outline of the proof of Proposition \ref{linkpest}.
	\begin{proof}
		We write $a= 1- 2\kappa v$. As in the proof of Proposition \ref{linwell}, we consider the approximate solution  $u_l(t)=\sum_{k=1}^l u_{l,k}(t) \phi_k$. 
		We differentiate the equation $i-1$ times ($i= 1, 2,\cdots, m$) with respect to $t$ to obtain  
		\begin{equation}\label{dtul}
			\lge a\partial^{i+1}_{t} u_l, w \rge +
			\sum_{j=2}^{i} C_{i,j}\lge  (\dt^{i+1-j} a)\dt^{j} u_l, w\rangle+
			\lge \partial^{i-1}_{t}\nabla u_l, \nabla w \rge + \lge \partial^{i}_{t} (-\Delta)^s u_l, w \rge
			= \lge \partial^{i-1}_{t}f, w \rge
		\end{equation}
		for any $w$ in the space spanned by the 
		eigenfunctions $\phi_1, \ldots, \phi_l$. (The sum on the left hand side does not appear when $i=1$.)
		
		\medskip
		\noindent\textbf{Step 1.} We choose $w= \partial^i_t u_l$ in (\ref{dtul}). 
		
		The smallness assumption on $v$ enables us to obtain the estimate
		$$\int^t_0\langle a(\tau)\partial^{i+1}_tu_l(\tau), \partial^{i}_tu_l(\tau)\rangle\diff\tau
		= \frac{1}{2}\langle a(t)\partial^{i}_tu_l(t), \partial^{i}_tu_l(t)\rangle-\frac{1}{2}
		\int^t_0\langle \partial_t a(\tau)\partial^{i}_tu_l(\tau), \partial^{i}_tu_l(\tau)\rangle\diff\tau$$
		$$\geq \frac{1}{4}\|\partial^i_t u_l(t)\|^2_{L^2}
		- C\int^t_0\|\partial^i_t u_l(\tau)\|^2_{L^2}\diff \tau.$$
		
		The key point is to estimate the term $\langle \partial^{i+1-j}_t a(\tau)\partial^{j}_tu_l(\tau), \partial^{i}_tu_l(\tau)\rangle$ in the sum in (\ref{dtul}).
		
		Note that if $i\leq m-2$, then $i+1-j\leq m-3$ for $j\geq 2$. In this case by Proposition \ref{emb} we have
		$$a\in \bigcap_{k=2}^{m-1} C^{m-k-1}([0,T]; C^{k-2}(\bar{\Omega}))\subset C^{i+1-j}([0,T]; C^{k-2}(\bar{\Omega})).$$
		Then we have the estimate
		\begin{equation}\label{m-2est}
			\int^t_0\langle \partial^{i+1-j}_t a(\tau)\partial^{j}_tu_l(\tau), \partial^{i}_tu_l(\tau)\rangle\diff\tau
			\leq C_a\int^t_0(\|\partial^{j}_tu_l(\tau)\|^2_{L^2}+ \|\partial^{i}_tu_l(\tau)\|^2_{L^2})\diff\tau,
		\end{equation}
		and we can use the argument similar to the one in the proof of Proposition \ref{linwell} to obtain 
		$$\|\partial^i_t u_l(t)\|_{L^2}^2
		+ \|\partial^{i-1}_t\nabla u_l(t)\|_{L^2}^2  + \int_0^t\|\partial^i_t (-\Delta)^{s/2} u_l(\tau)\|_{L^2}^2\diff \tau\\
		\leq C\|f\|_{Z^{m-1}}^2
		$$
		for $i\leq m-2$ and $t\in [0, T]$. In particular, we have
		\begin{equation}\label{m-2tpt}
			\int^T_0\|\partial^i_t u_l(t)\|_{L^2}^2\diff t
			+ \int^T_0|\partial^{i-1}_t\nabla u_l(t)\|_{L^2}^2\diff t \\
			\leq C'\|f\|_{Z^{m-1}}^2,\quad i\leq m-2.
		\end{equation}
		
		For the cases $i= m-1, i=m$, we cannot control the $L^\infty$-bound of $\partial^{m-2}_t a, \partial^{m-1}_t a$ so (\ref{m-2est}) will not work. We will deal with these two cases in Step 3.
		
		\medskip
		\noindent\textbf{Step 2.} Next we choose $w= \partial^{i-1}_t (-\Delta)^k u_l$ in (\ref{dtul}) to obtain higher regularity results.
		For $i\leq m-2$ and $k\leq m-i-2$, we can derive the estimate
		\begin{equation}\label{Hkest}
			\int^t_0\|\partial^{i-1}_tu_l(\tau)\|^2_{H^{k+1}}\diff\tau\leq C
			\int^t_0(\|\partial^{i+1}_tu_l(\tau)\|^2_{H^{k+1}}+ +\sum^i_{j=2}\|\partial^{j}_tu_l(\tau)\|^2_{H^{k}}+ \|\partial^{i-1}_tf(\tau)\|^2_{H^{k}})\diff\tau
		\end{equation}
		for $t\in [0, T]$ based on  (\ref{dtul}) and the $L^\infty$-boundedness of the derivatives of $a$.
		
		Let $k=1$ in (\ref{Hkest}). By (\ref{m-2tpt}) we have
		$$\int^t_0\|\partial^{i-1}_tu_l(\tau)\|^2_{H^{2}}\diff\tau\leq C\|f\|^2_{Z^{m-1}}.$$
		The assumption $m\geq 8$ and Proposition \ref{emb} ensure that 
		\begin{equation}\label{pt23bdd}
			\partial^{j}_tu_l\in H^1(0, T; H^2(\Omega))\hookrightarrow C([0, T]; C(\bar{\Omega})),\qquad j= 2,3.
		\end{equation}
		
		\medskip
		\noindent\textbf{Step 3.} For the case $i= m-1$, the first term in the sum in (\ref{dtul}) has the estimate
		$$\int^t_0\langle \partial^{m-2}_t a(\tau)\partial^{2}_tu_l(\tau), \partial^{m-1}_tu_l(\tau)\rangle\diff\tau$$
		$$\leq \frac{1}{C}\|\partial^{2}_tu_l\|_{L^\infty}\int^t_0\|\partial^{m-2}_t a(\tau)\|^2_{L^2}\diff \tau+ C\int^t_0\|\partial^{m-1}_t u_l(\tau)\|^2_{L^2}\diff \tau$$
		by (\ref{pt23bdd}). We similarly deal with the case $i= m$ based on the $L^\infty$-boundedness of 
		$\partial^{3}_tu_l$. Then we can show that (\ref{m-2tpt}) also works for $i= m-1$ and $i= m$, i.e.
		\begin{equation}\label{mtpt}
			\int^T_0\|\partial^i_t u_l(t)\|_{L^2}^2\diff t
			+ \int^T_0|\partial^{i-1}_t\nabla u_l(t)\|_{L^2}^2\diff t \\
			\leq C'\|f\|_{Z^{m-1}}^2,\quad i\leq m.
		\end{equation}
		
		\medskip
		\noindent\textbf{Step 4.} (\ref{mtpt})
		implies that 
		$$u_l\in H^m(0,T; L^2(\Omega))\cap H^{m-1}(0,T; H^1(\Omega)).$$
		Then we can inductively show that $u_l$ is bounded above by $C\|f\|_{\zmm}$ in $H^{k}(0,T; H^{m-k}(\Omega))$ for 
		$0\leq k\leq m$. In fact, suppose we have verified that the estimate holds for $k\geq K$, we can show that the estimate holds for $K-1$ based on the inductive hypothesis and (\ref{dtul}) when choosing $w= \partial^{i-1}_t (-\Delta)^{m-k}u_l$.
		
		Once we have the boundedness of $\{u_l\}^\infty_{l=1}$, we can use the standard compactness argument to conclude that a subsequence of $\{u_l\}^\infty_{l=1}$ weakly convergent to the solution $u$ of (\ref{linkp}).
		
	\end{proof}
	\medskip
	\bibliographystyle{plain}
	{\small\bibliography{Reference9}}
\end{document}